\documentclass[a4paper,10pt]{amsart}
\usepackage[english]{babel}
\usepackage[utf8]{inputenc}
\usepackage[T1]{fontenc}
\usepackage[style=numeric,%alphabetic,%
	useprefix,%
	giveninits=true,%
	hyperref,%
	doi=false,%
	url=false,%
	isbn=false,%
	backend=bibtex,%
	maxbibnames=99%
	]{biblatex}
\usepackage{amssymb}
\usepackage{mathrsfs}
\usepackage{hyperref}
\usepackage{soul}
\usepackage[usenames,dvipsnames]{xcolor}
\hypersetup{colorlinks,%
citecolor=Black,%
filecolor=Black,%
linkcolor=Black,%
urlcolor=Black}
\usepackage{enumitem}	% Per personalizzare gli elenchi
\allowdisplaybreaks % per avere equazioni su più pagine
\usepackage{enumitem} % per inserire testi nelle enumerazioni
\usepackage{yfonts}    % per comando textfrak
\newcommand{\scr}[1]{\mathscr{#1}}

\newcommand{\bb}[1]{\mathbb{#1}}
\newcommand{\cal}[1]{\mathcal{#1}}
\newcommand{\N}{\mathbb{N}}	% Numeri naturali
\newcommand{\R}{\mathbb{R}}	% Numeri reali
\newcommand{\dd}{d} %\,\mathrm{d}}	% 'd' di derivata
\newcommand{\de}{\partial}		% Derivata parziale
\newcommand{\HH}{\mathbb{H}}	% Heisenberg group
\newcommand{\sphH}{\scr S}		% Spherical Hausdorff measure
\newcommand{\coarea}{\cal C}	% Coarea factor
\newcommand{\cur}[1]{[\![#1]\!]}	% #1 made into a current
\newcommand{\spt}{\mathtt{spt}}	% support of a function, etc...
\newcommand{\W}{\mathbb{W}} % Vertical subspace
\newcommand{\hel}{\llcorner} % restriction
\newcommand{\V}{\bb V}	% horizontal homogeneous subgroup
\newcommand{\areaf}{\cal A}	% Area factor

\newcommand{\HM}{\mathcal{L}(\HH^n;\R^\ell)}	% homogeneous morphisms
\newcommand{\bwl}{\text{\Large$\wedge$}}
\renewcommand{\DH}{\mathcal D_\HH}
\newcommand{\Tcurr}{\mathsf T}

\theoremstyle{plain}
\newtheorem{proposition}{Proposition}[section]
\newtheorem{theorem}[proposition]{Theorem}
\newtheorem{lemma}[proposition]{Lemma}
\newtheorem{corollary}[proposition]{Corollary}
\newtheorem{thm}{Theorem}[section]

\theoremstyle{definition}
\newtheorem{definition}[proposition]{Definition}
\newtheorem{remark}[proposition]{Remark}
\theoremstyle{remark}

\title[Lipschitz functions on submanifolds of Heisenberg groups]{Lipschitz functions on submanifolds\\ of Heisenberg groups}
\date{\today}
\author[Julia]{Antoine Julia}
	\address[A.~Julia]{D\'epartment de Math\'ematiques d'Orsay, Universit\'e Paris-Saclay, 91405, Orsay, France}
	\email{antoine.julia@u-psud.fr}

\author[Nicolussi~Golo]{Sebastiano Nicolussi Golo}
	\address[Nicolussi Golo]{Department of Mathematics and Statistics, 40014 University of Jyväskylä, Finland}	\email{sebastiano.s.nicolussi-golo@jyu.fi}
        
\author[Vittone]{Davide Vittone}
	\address[D.~Vittone]{Dipartimento di Matematica ``T.Levi-Civita'', Università di Padova, via Trieste 63, 35121 Padova, Italy.}
	\email{davide.vittone@unipd.it}

\thanks{A.~J.~has been supported by the Simons Foundation grant 601941, GD.  
S.~N.~G.~has been supported by the Academy of Finland (grant 322898 ``Sub-Riemannian Geometry via  Metric-geometry and Lie-group Theory'',
grant 314172 ``Quantitative rectifiability in Euclidean and non-Euclidean spaces''). 
D.~V.~has been supported by  FFABR 2017 of MIUR (Italy) and by GNAMPA of INdAM (Italy). All three authors have been supported by the University of Padova STARS Project ``Sub-Riemannian Geometry and Geometric Measure Theory Issues: Old and New''.}


\begin{document}

\begin{abstract}
We study the behavior of Lipschitz functions on intrinsic $C^1$ submanifolds of Heisenberg groups: our main result is their almost everywhere tangential Pansu differentiability. We also provide two applications: a Lusin-type approximation of Lipschitz functions on $\HH$-rectifiable sets, and a coarea formula  on $\HH$-rectifiable sets that completes the program started in~\cite{JNGV}.
\end{abstract}

\maketitle

\setcounter{tocdepth}{1}
\phantomsection
\addcontentsline{toc}{section}{Contents}
\tableofcontents

%%%%%%%%%%%%%%%%%%%%%%%%%%%%%%%%%%%%%%%%%%%%%%%%%%%%%%%%%%%%%%%
%%%%%%%%%%%%%%%%%%%%%%%%%%%%%%%%%%%%%%%%%%%%%%%%%%%%%%%%%%%%%%%

\section{Introduction}

Analysis on and of rectifiable sets in Euclidean spaces is made possible
by a variety of results, among which some of the most essential are the
Rademacher Theorem, the extension theorem for Lipschitz functions and
Area and Coarea formulae, see e.g.~\cite{EvansGariepy}. Starting from 
the 90's, these topics have been studied also in non Euclidean spaces through the
notion of rectifiability in metric spaces introduced by L.~Ambrosio and B.~Kirchheim~\cite{MR1189747, AmbrosioKirchheim}. 
There are, however, interesting spaces to which this notion is not adapted. 
For instance, the first Heisenberg group $\HH^1$ is purely
$k$-unrectifiable for $k=2,3,4$~\cite[Theorem 7.2]{AmbrosioKirchheim}; similar phenomena occur in non Abelian Carnot groups and more generally in
sub-Riemannian manifolds. Fortunately,  in
the setting of Carnot groups intrinsic notions of rectifiability are available,   modeled either on intrinsic $C^1$ submanifolds or on the so-called intrinsic Lipschitz graphs~\cite{FS_JGA}. The two notions are in general different~\cite{JNGV_AntiRademacher} but they coincide~\cite[Corollary 7.4]{2020arXiv200714286V} in  Heisenberg groups $\HH^n$, where intrinsic rectifiable sets
are now relatively well understood
  and results analogue to those
mentioned above are known to hold 
\cite{antonelli2021rectifiable_representation,antonelli2021rectifiable_aaa_structure,
chousionis2018intrinsic,corni2021area,DDFO,
fassler2020semmes,JNGV,MagStepTrev,
merlo2019geometry,merlo2020marstrandmattila,NaorYoung,2020arXiv200714286V}.
%\Dnote{What else?}

We stress the fact that these results depend strongly on the particular
Carnot group one studies. This is in sharp contrast with the study of
 rectifiability in metric spaces, which strongly relies on the analytic properties
of the Euclidean spaces on  which metric rectifiable sets are modeled,
and not so much on the properties of the space itself. There are indeed
Carnot groups for which some results fail (e.g.~the extension and Rademacher theorems for intrinsic Lipschitz graphs \cite{AntonelliMerloUnextendable,JNGV_AntiRademacher}) or are
still unknown.

In this paper we go one step further towards the understanding of rectifiable sets in Heisenberg groups $\HH^n$. Our main result is a Rademacher-type
Theorem for Lipschitz functions defined on intrinsic $C^1$ submanifolds in $\HH^n$, see Theorem~\ref{thm607bf499} below; analogous versions for Lipschitz functions defined on intrinsic Lipschitz graphs or on $\HH$-rectifiable sets in $\HH^n$ are provided later in Section~\ref{sec_proofThmA}, see Corollaries~\ref{cor_RademacherLipgr} and~\ref{cor_RademacherHrectifiablesets}. 
We will consider only submanifolds and $\HH$-rectifiable sets {\em of low codimension $m\leq n$}; the other case {\em of low dimension} (i.e., of codimension more than $n$)  is more straightforward, as these objects turn out to have  standard Euclidean regularity in $\R^{2n+1}$~\cite{antonelli2020intrinsically}.

Before stating Theorem~\ref{thm607bf499}, we need to provide the notion of differentiability along a submanifold. Heisenberg groups  and $C^1_\HH$ submanifolds in $\HH^n$ will be introduced in Section~\ref{sec_prelim}. In the following, $d$ denotes a homogeneous distance on $\HH^n$.

\begin{definition}[Differentiability on a submanifold]
Let $S\subset\HH^n$ be a $C^1_\HH$ submanifold
 of codimension $m\leq n$; we say that a map $u:S\to\R^\ell$ is {\it tangentially Pansu differentiable along $S$ at $p\in S$}
(cfr.~\cite[Definition 2.89]{MR1857292}) %,page 98
if there exists a group morphism $L:\HH^n\to\R^\ell$ such that
\begin{equation}\label{eq_defdifferentiability}
\lim_{\substack{q\to p,\\ q\in S}} \frac{|u(q) - u(p) -  L(p^{-1}q)|}{d(p,q)} = 0.
\end{equation}
\end{definition}

The morphism $L$ for which~\eqref{eq_defdifferentiability} holds is, in general, not unique; however, it can be  proved that $L$ is uniquely determined on the tangent space $T^\HH_pS$. 
This uniqueness is a consequence of  statement \ref{item607d35ae} in Proposition~\ref{prop607d3565}, which is equivalent to tangential differentiability.
The restriction $L|_{T^\HH_pS}$ will be called {\it Pansu differential of $u$ at $p$ along $S$} and it will be denoted by $D_\HH^S u(p)$ or $D_\HH^S u_p$.

We can now state our main result; as customary, we denote by $Q=2n+2$ the homogeneous dimension of $\HH^n$, so that the Hausdorff dimension of a $C^1_\HH$ submanifold  of codimension $m\leq n$  is $Q-m$.

\begin{thm}[Pansu--Rademacher]\label{thm607bf499}
	Let $n,m,\ell$ be positive integers with $ m < n$.
	If $S$ is a $C^1_\HH$ submanifold of $\HH^n$ of codimension $m$
	and $u:S\to\R^\ell$ is a Lipschitz function,
	then 
	$u$ is tangentially Pansu differentiable at $\sphH^{Q-m}$-a.e.~point of $S$.
\end{thm}

Theorem~\ref{thm607bf499} is not trivial. It does not directly follow from the Pansu Theorem~\cite{Pansu} on the a.e.~differentiability of Lipschitz functions in $\HH^n$: in fact, a Lipschitz function $u:\HH^n\to\R^\ell$ could be {\em nowhere} differentiable on $S$. On the contrary, Theorem~\ref{thm607bf499} asserts that $u$ must be $\sphH^{Q-m}$-a.e. differentiable  along the horizontal directions that are tangent to $S$. 
In classical Euclidean geometry an analogous result can be easily obtained from
the usual Rademacher Theorem by reasoning in local charts on the submanifold. 
In Heisenberg groups $\HH^n$ a similar strategy seems feasible only for submanifolds of codimension 1   with  stronger $C^{1,\alpha}_\HH$ regularity, because  these submanifolds can be modeled on the Carnot group $\HH^{n-1}\times\R$ (see~\cite[Theorem~1.7]{DDFO}) where Pansu Theorem holds.

Our approach is completely different: Theorem~\ref{thm607bf499} is in fact proved via the use of currents in the Heisenberg group (see Section~\ref{sec_prelim}): although these currents involve the use of Rumin's complex of differential forms, whose construction is highly non-trivial, our proof does not require its most daunting aspects. Let $\cur S$ be the current associated with the submanifold $S$ and  without loss of generality assume that $\ell=1$. We consider the blow-up of the current $u\cur S$ at a point $p\in S$ and prove that, for $\sphH^{Q-m}$-a.e.~$p\in S$, the blow-up limit is of the form $L\cur{T^\HH_pS}$, where $T^\HH_pS$ is the homogeneous tangent subgroup to $S$ at $p$ and $L$ is a homogeneous morphism $L:T^\HH_pS\to\R$. Through some minor technicalities (see Proposition~\ref{prop607d3565} and Lemma~\ref{lem607d9517}), this fact implies the tangential differentiability of $u$ along $S$ at $p$.

We must stress the fact that, in Theorem~\ref{thm607bf499}, the assumption that the codimension $m$ is {\em strictly} less than $n$ is crucial, as the following example shows.

\begin{remark}\label{rem_counterexample_m=n}
Consider the $C^1_\HH$ submanifold $S:=\{(x,y,t)\in\HH^1\equiv\R^3:x=0\}$ of codimension 1 in $\HH^1$ and let $u:S\to\R$ be the function $u(0,y,t):=v(t)$, where $v:\R\to\R$ is a $\tfrac12$-H\"older continuous function such that, for every $t\in\R$, 
\[
\liminf_{s\to t}\frac{|v(s)-v(t)|}{|s-t|^{1/2}}>0.
\]
For the construction of such a $v$, see e.g.~\cite[Appendix]{JNGV_AntiRademacher} and the references therein. The H\"older continuity of $v$ easily implies the Lipschitz continuity of $u$ on $S$ with respect to the distance $d$. Now, every group morphism $L:\HH^1\to\R$ is such that $L(0,0,t)=0$;  taking into account that $S$ is an Abelian subgroup of $\HH^1$ (as a group, it is isomorphic to $\R^2$) we deduce that for every fixed $(0,y,t)\in S$
\[
\liminf_{s\to t}\frac{|u(0,y,s)-u(0,y,t)-L((0,y,t)^{-1}(0,y,s))|}{d((0,y,s),(0,y,t))}= c\liminf_{s\to t}\frac{|v(s)-v(t)|}{|s-t|^{1/2}} >0,
\]
where the constant $c>0$ depends on the distance $d$. In particular, there is no group morphism $L$ for which~\eqref{eq_defdifferentiability} holds, and $u$ is a Lipschitz function that is {\em nowhere} tangentially Pansu differentiable along $S$.
\end{remark}

We conclude this introduction by stating two consequences of Theorem~\ref{thm607bf499}. The first one is a Lusin-type theorem for Lipschitz functions on $\HH$-rectifiable sets: a Lipschitz function coincide with a $C^1_\HH$ function outside  an arbitrarily small set. The tangential Pansu differential along a $\HH$-rectifiable subset, $D_\HH^Ru_p$,  
is introduced in Corollary~\ref{cor_RademacherHrectifiablesets}.

\begin{thm}[Lusin]\label{thm607bf4ef}
	Let $n,m,\ell\ge1$ with $ m < n$.
	%\Dnote{Should we call it Whitney-Lusin, or similar?}
	Let $R$ be a $\HH$-rectifiable subset of $\HH^n$ with codimension $m$
	and $u:R\to\R^\ell$ a Lipschitz function.
	For every $\epsilon>0$ there is $g\in C^1_\HH(\HH^n;\R^\ell)$ such that 
	\[
	\sphH^{Q-m}(\{p\in R: u(p)\neq g(p)\text{ or }D_\HH^Ru_p\neq D_\HH^Rg_p\}) < \epsilon .
	\] 
	Moreover, $g$ can be chosen to be Lipschitz continuous on $\HH^n$ with a Lipschitz constant controlled only in terms of $n$ and of the Lipschitz constant of $u$.
\end{thm}

A second consequence of Theorem~\ref{thm607bf499} is a fully general coarea formula on $\HH$-rectifiable sets, Theorem~\ref{thm607bf4fe}.
In our previous work~\cite{JNGV} we proved a coarea formula under the assumption that the ``slicing'' function $u$ is of class $C^1_\HH$; the use of Theorem~\ref{thm607bf4ef} allows to extend this result to the more general (and more natural) case in which $u$ is Lipschitz continuous. 
In fact, our interest in Theorem~\ref{thm607bf499} was originally motivated by Theorem~\ref{thm607bf4fe}, which completes the program started in~\cite{JNGV} at least in Heisenberg groups.

\begin{thm}[Coarea]\label{thm607bf4fe}
	Let $n,m,\ell\ge1$ with $ m+\ell \le n$.
	There is a continuous positive function $\coarea(\bb P,\alpha)$, 
	defined for homogeneous subgroups $\bb P$ of $\HH^n$ of codimension $m$
	and homogeneous group morphisms $\alpha:\bb P\to\R^\ell$,
	such that the following holds.
	If $R$ and $u$ are as in Theorem~\ref{thm607bf4ef},
	then, for every Borel function $h:R\to[0,+\infty)$,
	\[
	\int_R h(p) \coarea(T^\HH_pR,D_\HH^Ru_p) \dd\sphH^{Q-m}(p)
	= \int_{\R^\ell} \int_{u^{-1}(s)} h(x) \dd\sphH^{Q-m-\ell}(x) \dd\mathscr L^\ell(s) .
	\]
	Moreover, if the distance $d$ is rotationally invariant\footnote{See~\eqref{eq:seba1} for the definition of rotationally invariant distance.}, then
	then there exists a constant $\textfrak c=\textfrak c(n,m,\ell,d)>0$ such that
		\begin{equation*}
	\textfrak c \int_R h(p)J^R_H u(p) \, \dd\sphH^{Q-m} (p)=\int_{\R^\ell} \int_{u^{-1}(s)} h(x)\dd\sphH^{Q-m-\ell}(x)\,\dd\mathscr L^\ell(s)
		\end{equation*}
	where
	\[
	      J^R_H u(p) = (\det(L \circ L^T))^{1/2} \quad \text{ with } \quad L= D_\HH^Ru_p \vert_{T^\HH_pR}.
	\]
\end{thm}
\medskip

The paper is structured as follows. Section~\ref{sec_prelim} contains the preliminary material about Heisenberg groups, $C^1_\HH$ submanifolds, $\HH$-rectifiable sets and currents, while Section~\ref{sec_differentiability} is concerned with some technical results about tangential Pansu differentiability. Theorems~\ref{thm607bf499},~\ref{thm607bf4ef} and~\ref{thm607bf4fe} are eventually proved in Sections~\ref{sec_proofThmA},~\ref{sec_proofThmB} and~\ref{sec_proofThmC}, respectively.
\medskip

{\em Acknowledgments.}
During the preparation of this paper we were informed that Theorem~\ref{thm607bf499} also follows from some results contained in a forthcoming paper by  G.~de~Philippis, A.~Marchese, A.~Merlo, A.~Pinamonti and F.~Rindler: their method, which follows the approach in~\cite{AlbertiMarchese}, is easier to generalize to other Carnot groups, though possibly less hands-on than ours. We warmly thanks them for sharing this information with us.

\section{Preliminaries}\label{sec_prelim}
For an integer $n\ge 1$, the $n$-th {\em Heisenberg group} $\HH^n$ is the nilpotent, connected and simply connected stratified Lie group associated with the step 2 algebra $V=V_1\oplus V_2$ defined by
\begin{align*}
& V_1=\textrm{span}\{X_1,\dots,X_n,Y_1,\dots,Y_n\},\qquad V_2=\textrm{span}\{T\}
\end{align*}
and where the only non-vanishing commutation relations are given by $[X_i,Y_i]=T$  for every $i=1,\dots,n$.
We will always identify $\HH^n$ with its Lie algebra through the exponential  map $\exp:V\to\HH^n$.
This  induces a diffeomorphism between $\HH^n$ and $\R^{2n+1}$ defined by
\[
\R^n\times \R^n\times \R\ni (x,y,t)\longleftrightarrow \exp(x_1X_1+\dots+x_nX_n+y_1Y_1+\dots+ y_nY_n+tT)\in\HH^n
\]
according to which the group operation reads
\[
(x,y,t)(x',y',t')=(x+x',y+y',t+t'+\tfrac12\textstyle\sum_{j=1}^n(x_jy_j'-x_j'y_j)).
\]
In these coordinates the generators of the algebra read as
\[
X_i=\partial_{x_i}-\frac{y_i}2 \partial_t,\qquad Y_i=\partial_{y_i}+\frac{x_i}2\partial_t,\qquad T=\partial_t
\]
for every $i=1,\dots,n$. In particular, the space $V_1$ is the kernel of the left-invariant {\em contact form} $\theta:=dt+\frac12\sum_{i=1}^n(y_idx_i-x_idy_i)$.

Heisenberg groups are endowed with dilations, i.e., with the one-parameter group of automorphisms $(\delta_\lambda)_{\lambda>0}$ defined by $\delta_\lambda(x,y,t):=(\lambda x,\lambda y,\lambda^2t)$. We endow $\HH^n$ with a  left-invariant and homogeneous distance $d$, so that
\[
d(p,q)=d(p'p,p'q)\quad\text{and}\quad d(\delta_\lambda p,\delta_\lambda q)=\lambda d(p,q)\qquad\text{for every }p,p',q\in\HH^n,\lambda>0,
\] 
and denote by $B(p,r)$ the open ball of center $p\in\HH^n$ and radius $r>0$.
The Hausdorff dimension of $\HH^n$ is $Q:=2n+2$. 

We fix on $V$ the scalar product making the basis $X_1,\dots,X_n,Y_1,\dots,Y_n,T$ orthonormal; for every $k\in\{0,\dots,2n+1\}$  a scalar product is canonically induced on the exterior product $\bwl_k V$. We will denote by $|\cdot|$ the norm associated with such scalar products. Also the dilations $\delta_\lambda$ can be canonically extended to $\bwl_k V$.

Given an open set $U\subset\HH^n$, we say that $f:U\to\R$ is  of class $C^1_\HH$ if $f$ is continuous and its horizontal derivatives 
\[
\nabla_\HH f:=(X_1f,\dots,X_nf,Y_1f,\dots,Y_nf)
\]
are represented by continuous functions on $U$. In this case we write $f\in C^1_\HH(U)$. We agree that, for every $p\in U$,  $\nabla_\HH f(p)\in\R^{2n}$ is identified with the horizontal vector
\[
\nabla_\HH f(p):=X_1f(p)X_1+\dots+Y_nf(p)Y_n \in V_1
\]
We denote by $ C^1_\HH(U,\R^m)$ the space of functions $f:U\to \R^m$ whose components belong to $C^1_\HH(U)$.

\begin{definition}\label{def_C1Hsubmanifold}
Let $m\in\{1,\dots,n\}$ be fixed. We say that $S\subset \HH^n$ is a submanifold {\em of class} $C^1_\HH$ (or  {\em $\HH$-regular submanifold}) of codimension $m$ if, for every $p\in S$, there exist an open neighborhood $U\subset\HH^n$ of $p$ and $f\in C^1_\HH(U,\R^m)$ such that
\[
S\cap U=\{q\in U:f(q)=0\}\quad\text{and}\quad \text{$\nabla_\HH f(q)$ has rank $m$ for all $q\in U$.}
\] 
We also define the {\em horizontal normal} $n_S^\HH(p)$ to $S$ at $p$ as the horizontal $m$-vector  
\[
n_S^\HH(p):=\frac{\nabla_\HH f_1(p)\wedge\dots\wedge\nabla_\HH f_m(p)}{|\nabla_\HH f_1(p)\wedge\dots\wedge\nabla_\HH f_m(p)|}\in\bwl_m V_1
\]
and the {\em (horizontal) tangent} $t^\HH_S(p):=*n_S^\HH(p)\in \bwl_{2n+1-m} V$.\\
We will  consider the {\em boundary} of $S$ defined as
  $\partial  S:=\overline S \setminus S$.
\end{definition}

In the definition of the tangent multi-vector $t^\HH_S$ the symbol $*$ denotes the Hodge operator from multivector calculus. 
 It is well known that the blow-up limit  of a $C^1_\HH$ submanifold $S$ at $p\in S$ is the homogeneous (i.e., dilation-invariant) subgroup
\[
T^\HH_pS:=\exp(\{X\in V:X\wedge t^\HH_S=0\}).
\]
This means in particular that $\lim_{\lambda\to+\infty} \delta_{1/\lambda}(p^{-1}S)=T^\HH_pS$ in the sense of Kuratowski, see % e.g.
Section~\ref{sec_differentiability}. We will refer to $T^\HH_pS$ as  the {\em homogeneous tangent space} (or simply {\em  tangent space}) to $S$ at $p$.

An Implicit Function Theorem~\cite[Theorem~6.5]{MR1871966} is available for $C^1_\HH$ submanifolds. If $S$ is as in Definition~\ref{def_C1Hsubmanifold} and $p\in S$ is fixed, then there exist
\begin{itemize}
\item a {\em horizontal complement} $\V=\V(p)$ to $T^\HH_pS$, i.e., a homogeneous subgroup $\V$ such that $\V\subset V_1$, $\V\cap T^\HH_pS=\{0\}$ and $\HH^n=(T^\HH_pS)\cdot\V$;
\item an open neighborhood $\Omega$ of $p$;
\item a relatively open set $U\subset T^\HH_pS$;
\item a continuous map $\phi:U\to\V$ 
\end{itemize}
such that $S\cap\Omega$ coincides with the {\em intrinsic graph} $\Gamma_\phi$ of $\phi$ defined by
\begin{equation}\label{eq_intrinsicgraph}
\Gamma_\phi:=\{w\phi(w):w\in U\}.
\end{equation}
See e.g.~\cite{JNGV} and the references therein. 
The area formula for such graphs states that there exists a continuous function $\cal A_\phi:U\to(0,+\infty)$ such that for every Borel function $h:S\to[0,+\infty)$
\begin{equation}\label{eq_areaformula}
\int_{S\cap \Omega}h\dd\sphH^{Q-m}=\int_U h(w\phi(w))\cal A_\phi(w)\dd\sphH^{Q-m}(w).
\end{equation}
Recall that the Hausdorff dimension of $S$ (as well as that of $T^\HH_pS$) is $Q-m$; moreover, the spherical Hausdorff measure $\sphH^{Q-m}$ is locally $(Q-m)$-Ahlfors regular on $S$.

\begin{remark}\label{rem_areafactoris1}
We will later use the fact that, if $\bar w\in T^\HH_pS$ is the unique point such that $p=\bar w\phi(\bar w)$, then $\cal A_\phi(\bar w)=1$.  This follows from the very definition of the area factor $\cal A$ for the spherical measure $\sphH^{Q-m}$, see~\cite[Lemma~3.2]{JNGV}.
\end{remark}

\begin{definition}\label{def_Hrectifiable}
Let $m\in\{1,\dots,n\}$ be fixed. We say that $R\subset \HH^n$ is  {\em countably $\HH$-rectifiable} of codimension $m$ if there exist countably many $C^1_\HH$ submanifolds $S_i$, $i\in\N$, of codimension $m$ such that
\[
\sphH^{Q-m}\Big(R\setminus\bigcup_{i\in\N} S_i\Big)=0.
\]
We say that $R$ is  {\em  $\HH$-rectifiable} if, in addition, $\sphH^{Q-m}(R)<+\infty$.
\end{definition}

The following  lemma, though very simple, is sometimes overlooked. 

\begin{lemma}\label{lem_oneSisenough}
Let $m\leq n$ be fixed. Then, a subset $R\subset\HH^n$ is  $\HH$-rectifiable  of codimension $m\leq n$ if and only if, for every  $\varepsilon>0$, there exists a $C^1_\HH$ submanifold $S\subset\HH^n$ of codimension $m$ such that
\begin{equation}\label{eq_unasolasuperficie}
\sphH^{Q-m}(R\setminus S)<\varepsilon.
\end{equation}
\end{lemma}
\begin{proof}
Let $\varepsilon>0$ be fixed and fix $S_i$, $i\in\N$,  as in Definition~\ref{def_Hrectifiable}. 
Fix also a positive integer $M$ such that
\[
\sphH^{Q-m}\Big(R\setminus\bigcup_{i\leq M} S_i\Big)<\frac\varepsilon2.
\]
We define the $C^1_\HH$ submanifold $S_0':=\{p\in S_0: d(p,\de S_0)>r_0\}$, where  $r_0$ is chosen so that
\[
\sphH^{Q-m}(R\cap\de S_0')=0\qquad\text{and}\qquad \sphH^{Q-m}((R\cap S_0)\setminus S_0')<\frac\varepsilon4.
\]
Reasoning by induction, for every $i=1,\dots,M$ one can define $C^1_\HH$ submanifolds  
\[
S_i':=\{p\in S_i\setminus \cup_{j<i}\overline{S_j'}: d(p,\de (S_i\setminus \cup_{j<i}\overline{S_j'}))>r_i\}
\]
where we use the fact that $S_i\setminus \cup_{j<i}\overline{S_j'}$ is a $C^1_\HH$ submanifold  and $r_i>0$ is chosen so that
\[
\sphH^{Q-m}(R\cap\de S_i')=0\qquad\text{and}\qquad \sphH^{Q-m}(R\cap(S_i\setminus \cup_{j<i}\overline{S_j'})\setminus S_i')<\frac\varepsilon{2^{i+2}}.
\]
We now consider $S:=\cup_{i=0}^M S_i'$, which is a $C^1_\HH$ submanifold because it is  union of finitely many $C^1_\HH$ submanifolds at positive distance from each other. Then
\begin{align*}
\sphH^{Q-m}(R\setminus S)
& \leq  \sphH^{Q-m}( R\setminus\cup_{i\leq M} S_i) + \sphH^{Q-m}( R\cap(\cup_{i\leq M} S_i)\setminus(\cup_{j\leq M} S_j'))\\
&< \frac\varepsilon2 + \sphH^{Q-m}(\cup_{i\leq M} ((R\cap S_i)\setminus\cup_{j\leq M} S_j'))\\
&\leq \frac\varepsilon2 + \sphH^{Q-m}(\cup_{i\leq M} (R\cap(S_i\setminus \cup_{j\leq i}{S_j'})))\\
&= \frac\varepsilon2 + \sphH^{Q-m}(\cup_{i\leq M} (R\cap(S_i\setminus \cup_{j<i}\overline{S_j'})\setminus S_i'))\\
&< \varepsilon,
\end{align*}
where we used the fact that $\sphH^{Q-m}(R\cap\de S_j')=0$. This proves one implication, the converse one  is trivial.
\end{proof}

\begin{definition}\label{def60af5092}
	An {\em approximate tangent space} $T^\HH_p R$ can be defined for a countably $\HH$-rectifiable set $R\subset \HH^n$. 
	Let   $S_i$ be as in Definition~\ref{def_C1Hsubmanifold}; then we define
	\[
	T^\HH_pR := T^\HH_p S_i\qquad\text{if }p\in R\cap S_i\setminus\bigcup_{j<i}S_j.
	\]
\end{definition}
Definition~\ref{def60af5092} is well-posed $\sphH^{Q-m}$-a.e. on $R$, see e.g.~\cite[\S2.5]{JNGV}. It turns out that, if $R_1,R_2\subset\HH^n$ are countably $\HH$-rectifiable, then $T^\HH_p R_1=T^\HH_p R_2$ for $\sphH^{Q-m}$-a.e. $p\in R_1\cap R_2$.

We will need a few facts from Rumin's  theory of differential forms in $\HH^n$ as well as from the theory of the associated currents. 
The exact complex  of {\em Heisenberg differential forms}
\[
0\to\R\to\Omega_\HH^0\stackrel{d_c}{\to}\Omega_\HH^1\stackrel{d_c}{\to}\dots\stackrel{d_c}{\to}\Omega_\HH^n\stackrel{d_c}{\to}\Omega_\HH^{n+1}\stackrel{d_c}{\to}  \dots   \stackrel{d_c}{\to}\Omega_\HH^{2n+1}\to 0
\]
was introduced by M.~Rumin  in~\cite{Rumin}; here we will only partially introduce  it and, for more details,  we refer  to~\cite[\S3]{2020arXiv200714286V} and the references therein. 
For $k\geq n+1$ we have
\[
\Omega_\HH^k:=\{\omega\text{ smooth $k$-form on }\HH^n: \omega\wedge\theta=\omega\wedge\dd\theta=0\},
\]
and $d_c:\Omega_\HH^k\to\Omega_\HH^{k+1}$ coincides with the usual exterior differential $d$.
Notice that $d\theta=-\sum_{j=1}^n dx_j\wedge dy_j$ is the standard symplectic form in $\R^{2n}$ (up to a sign).

For every $p\in\HH^n$, $\lambda>0$ and $\omega\in \Omega_\HH^k$, $k\geq n+1$, one has
\begin{equation}\label{eq_commute}
d(\omega\circ L_{p,\lambda})=\lambda(d\omega)\circ L_{p,\lambda}, \qquad\text{where }L_{p,\lambda}(x)=\delta_\lambda(px).
\end{equation}
where, by a slight abuse of notation, we identify $k$-differential forms with functions $\HH^n\to\bwl^kV$.
Formula~\eqref{eq_commute} can be  proved on observing that, by definition of the Rumin's spaces,
one can write $\omega=\omega_H\wedge\theta$ for a suitable $\omega_H\in C^\infty(\HH^n,\bwl^{k-1}V_1)$ such that $\omega_H\wedge d\theta=0$; in this way
\[
d\omega=d(\omega_H\wedge\theta)=(d\omega_H)\wedge \theta=(d\omega_H)_H\wedge \theta
\]
for a suitable $(d\omega_H)_H\in C^\infty(\HH^n,\bwl^{k}V_1)$, and we obtain the  homogeneity relations 
\[
\omega\circ L_{p,\lambda}=\lambda^{-k-1}L_{p,\lambda}^*\omega,\qquad (d\omega)\circ L_{p,\lambda}=\lambda^{-k-2}L_{p,\lambda}^*(d\omega),
\]
where $L_{p,\lambda}^*$ denotes pull-back by $L_{p,\lambda}$. Since pullback and exterior differentiation commute, we eventually achieve
\[
d(\omega\circ L_{p,\lambda})=d(\lambda^{-k-1}L_{p,\lambda}^*\omega)=\lambda\  \lambda^{-k-2}L_{p,\lambda}^*(d\omega)= \lambda (d\omega)\circ L_{p,\lambda}.
\]

Let $\DH^k\subset\Omega_\HH^k$ be the space of Heisenberg $k$-forms with compact support; $d_c$ maps $\DH^k$ to $\DH^{k+1}$. A {\em Heisenberg $k$-current} is, by definition, an element of the dual space to  $\DH^k$. If $S\subset\HH^n$ is a $C^1_\HH$ submanifold of codimension $m\leq n$ with $\sphH^{Q-m}\hel S$ locally finite, then $S$  induces a Heisenberg $(2n+1-m)$-current $\cur{S}$ defined by
\[
\cur{S}(\omega)=\int_S \langle t^\HH_S(p)|\omega(p)\rangle \dd\sphH^{Q-m}(p),\qquad\omega\in\DH^{2n+1-m}.
\]
Observe that by definition $\cur S=t_S^\HH\sphH^{Q-m}\hel S$ where, given a Radon measure $\mu$ and a $\mu$-measurable function 
$t:\HH^n\to\bwl_k V$, we denote by $t\mu$ the Heisenberg  $k$-current 
\[
(t\mu)(\omega) = \int \langle t(p)|\omega(p) \rangle \dd\mu(p).
\]

The {\em boundary}  of a Heisenberg $k$-current $\Tcurr$ is the Heisenberg $(k-1)$-current $\de_c\Tcurr$  defined by
\[
\de_c\Tcurr(\omega)=\Tcurr(d_c\omega),\qquad\omega\in\DH^{k-1}.
\]

\begin{remark}\label{rem_boundarylocally0}
If $S\subset\HH^n$ is a $C^1_\HH$ submanifold of codimension $m\leq n$, then $\de_c\cur S=0$ locally on $S$, i.e., for every $p\in S$ there exists $r>0$ such that $\de_c\cur S(\omega)=0$ for every $\omega\in\DH^{2n-m}$ with support in $B(p,r)$. 
Indeed, 
$S$ locally coincides with an entire intrinsic Lipschitz graph on $T^\HH_pS$ by \cite[Theorem~1.5]{2020arXiv200714286V},
and the  currents canonically associated with entire intrinsic Lipschitz graphs have null boundary by~\cite[Proposition~7.5]{2020arXiv200714286V}.
\end{remark}

%%%%%%%%%%%%%%%%%%%%%%%%%%%%%%%%%%%%%%%%%%%%%%%%%%%%%%%%%%%%%%%
\section{Pansu differentiability on \texorpdfstring{$C^1_\HH$}{} submanifolds}\label{sec_differentiability}
Before stating and proving the following Proposition~\ref{prop607d3565} we need to fix some terminology. 
A sequence  $\{E_j\}_j$ of subsets of a topological space $X$ converges to $E\subset X$ in the sense of Kuratowski if the following two conditions are satisfied:
\begin{enumerate}
\item
if $x\in E$, then there exist $x_j\in E_j$ such that $x_j\to x$;
\item
if there are $j_k\to\infty$ and $x_k\in E_{j_k}$ such that $x_k\to x$, then $x\in E$.
\end{enumerate}
Accordingly, we say that a one-parameter family $\{E_\lambda\}_{\lambda\ge 1}$ of subsets of  $X$ converges to $E$ in the sense of Kuratowski
if, for every sequence $\lambda_j\to\infty$, the sequence $E_{\lambda_j}$ converges to $E$ in the sense of Kuratowski.

In a boundedly compact metric space $X$, Kuratowski limits satisfy standard properties:
the limit set $E$ is always sequentially closed;
the family of compact subsets  contained in a fixed bounded set is compact and, within  this family, Hausdorff convergence is equivalent to Kuratowski convergence;
every sequence of closed sets admits a convergent subsequence (cfr.~\cite[Mrowla’s Theorem, p.149]{MR1269778}).

We can now state the following result.

\begin{proposition}\label{prop607d3565}
	Let $S$ be a $C^1_\HH$ submanifold of $\HH^n$ of codimension $m\leq n$
	and let $u:S\to\R^\ell$ be a  function.
	Fix $p\in S$ and a homogeneous morphism $L:T^\HH_pS\to\R^\ell$.
	The following statements are equivalent:
	\begin{enumerate}[label=(\arabic*)]
		\item\label{item607f44f1}
	$u$ is tangentially Pansu differentiable along $S$ at $p$ and $D_\HH^Su_p=L$;
	\item\label{item607d35ae}
	The sets 
	\[
	\{(\delta_\lambda(p^{-1}x), \lambda(u(x)-u(p)) ) : x\in S \} \subset \HH^n\times\R^\ell
	\]
	converge to
	\[
	\{(x,L(x)):x\in T^\HH_pS\}
	\]
	in the sense of Kuratowski, as $\lambda\to\infty$.
	\item\label{item607d35b3}
	Let $U\subset T^\HH_pS$ be an open neighborhood of $0$ and $\phi:U\to\V$ (where $\V\subset V_1$ is a horizontal complement to $T^\HH_pS$) be such that $\Gamma_\phi=\{w\phi(w):w\in U\}\subset p^{-1}S$.
	Let $\phi_{\lambda}(w) := \delta_\lambda \phi(\delta_{1/\lambda}w)$; in particular, $\Gamma_{\phi_\lambda}=\delta_\lambda(\Gamma_\phi) \subset \delta_\lambda(p^{-1}S)$ 
	and $\phi_\lambda\to 0$ uniformly on compact sets.
	Then, the functions $v_\lambda:\delta_\lambda(U)\to \R^\ell$
	\[
	v_\lambda(w) := \lambda(u(p\delta_{1/\lambda}(w\phi_\lambda(w)))-u(p))
	\]
	converge uniformly on compact sets to $L$, as $\lambda\to\infty$.
\end{enumerate}
If, moreover, $u$ is Lipschitz continuous, the previous statements are equivalent to the following one:
\begin{enumerate}[resume,label=(\arabic*)]
	\item\label{item607d35b9}
	If $\tilde u:\HH^n\to\R^\ell$ is a Lipschitz extension of $u$, then $\tilde u|_{T^\HH_pS}$ is Pansu differentiable (as a map between homogeneous groups) at $0$ with differential $L$.
\end{enumerate}
\end{proposition}
\begin{proof}
	Without loss of generality,
	we assume $p=0$ and $u(0)=0$.
	The equivalence of~\ref{item607f44f1} and~\ref{item607d35ae} is an easy exercise.
	Next, notice that, for any neighborhood $\Omega\subset\HH^n\times\R^\ell$ of $(0,0)$ 
	and for $\lambda$ large enough, 
	\[
	\{(\delta_\lambda(x), \lambda u(x) ) : x\in S \}\cap\Omega
	= \{(w\phi_\lambda(w) , v_\lambda(w)) : w\in \delta_\lambda U\} \cap \Omega.
	\]
	Therefore, \ref{item607d35ae} and \ref{item607d35b3} are equivalent.
	
		Finally, we show that \ref{item607d35b3} is equivalent to \ref{item607d35b9} in case $u$ is Lipschitz continuous.
	The Pansu differentiability of $\tilde u|_{T^\HH_0S}$ at $0$ with differential $L$ 
	is equivalent to the locally uniform convergence of $\tilde u_\lambda(x) := \lambda \tilde u(\delta_{1/\lambda}x)$ to $L(x)$, for every $x\in T^\HH_0S$, as $\lambda\to\infty$.
	Notice that, if $C$ is a Lipschitz constant for $\tilde u$, then
	\begin{align*}
	|\tilde u_\lambda(w) - v_\lambda(w)|
	&= \lambda |\tilde u(\delta_{1/\lambda}w) - u(\delta_{1/\lambda}(w\phi_\lambda(w)) | \\
	&\le C\lambda d(\delta_{1/\lambda}w , \delta_{1/\lambda}(w\phi_\lambda(w))) \\
	&= C d(0,\phi_\lambda(w)) .
	\end{align*}
	Since $\phi_\lambda(w)\to 0$ locally uniformly, we conclude that $\tilde u_\lambda\to L$ if and only if $v_\lambda\to L$, as $\lambda\to\infty$.
\end{proof}

Before proving the next technical lemma let us fix some notation. 
Given $q\in\HH^n\equiv V_1\oplus V_2$ we denote by $q_H\in V_1$ the unique element such that $q-q_H\in V_2$. Recall that a scalar product $\cdot$ has been fixed on $V$. It is well-known that, if $\W\subset\HH^n$ is a homogeneous subgroup of codimension $m\leq n$ and $L:\W\to\R$ is a homogeneous morphism, then there exists a unique $v\in \W\cap V_1$ such that
\[
L(q)= v\cdot q_H\quad\text{for every }q\in\W.
\]
In case $\W=T^\HH_pS$ for some $C^1_\HH $ submanifold $S$ and $L=D_\HH^Su_p$ is the tangential Pansu differential along $S$ at $p\in S$ of some $u:S\to\R$, the vector $v$ introduced before is called {\it horizontal gradient along $S$} of $u$ at $p$ and it is denoted by $\nabla_\HH^Su(p)\in T^\HH_pS$. Observe that $\nabla_\HH^Su$ can be interpreted as a $V_1$-valued map defined on the set of tangential Pansu differentiability points along $S$ of $u$.

\begin{lemma}\label{lem_Borel}
	Let $S$ be a $C^1_\HH$ submanifold of $\HH^n$ of codimension $m\leq n$
	and let $u:S\to\R$ be a Borel function. Then
\begin{enumerate}
\item[(i)] the set $D\subset S$ of points where $u$ is tangentially Pansu differentiable along $S$ is a Borel set;
\item[(ii)] the map $\nabla_\HH^S u:D\to V_1$ is Borel.
\end{enumerate}
\end{lemma}

\begin{proof}
Let $L_k$, $k=1,2,\dots $ be a dense family of morphisms $\HH^n\to \R $. The set of differentiability points $D$ can be written as
  \begin{align*}
    D &= \left \{ p \in S: \exists L:\HH^n\to \R\text{ s.t. }\forall \epsilon >0\  \lim_{r\to 0}\sup_{q\in B(p,r)\cap S} \dfrac{| u(q)-u(p) - L(p^{-1}q)|}{d(p,q)} < \epsilon \right \}\\
        &= \bigcap_{j=1}^{\infty} \bigcup_{k=1}^\infty \left \{ p \in S: \lim_{r\to 0}  \sup_{q\in B(p,r)\cap S} \dfrac{| u(q)-u(p) - L_k(p^{-1}q)|}{d(p,q)} < \dfrac{1}{j}\right \}.
  \end{align*}
  Hence $D$ is Borel. To prove that the  horizontal gradient along $S$ is a Borel map,
 let $A$ be a closed subset of $V_1$ and let $v_k$, $k=1,2,\dots$ be a dense countable subset of $A$.
  There holds
  \begin{align*}
&   \{p\in S: \nabla_\HH^Su_p \in A\}  \ =\ \{p\in S: \nabla_\HH^Su_p \in \overline A\}\\
=&\bigcap_{j=1}^{\infty} \bigcup_{k=1}^\infty \left \{ p \in S: \lim_{r\to 0}  \sup_{q\in B(p,r)\cap S} \dfrac{| u(q)-u(p) -v_k \cdot (p^{-1}q)_H|}{d(p,q)} < \dfrac{1}{j}\right \},
  \end{align*}
  so that the map $p\mapsto \nabla_\HH^S u_p$ is Borel measurable on $D \subset S$.
\end{proof}

%%%%%%%%%%%%%%%%%%%%%%%%%%%%%%%%%%%%%%%%%%%%%%%%%%%%%%%%%%%%%%%
\section{Proof of Theorem~\ref{thm607bf499}}\label{sec_proofThmA}
In the following lemma, as well as in the sequel, limits of currents are understood with respect to the standard  weak-* topology on the space of currents, i.e., $\Tcurr_j\to\Tcurr$ if and only if $\Tcurr_j(\omega)\to\Tcurr(\omega)$ for every test Heisenberg form $\omega$. 
Moreover, given a $C^1_\HH$-submanifold $S$ of codimension $m\leq n$
and a function $u:S\to\R$, locally integrable with respect to $\sphH^{Q-m}\hel S$, we denote by $u\cur S$ the $(2n+1-m)$-Heisenberg current
\[
(u\cur{S})(\omega) 
:= 
\int_S u\: \langle t^\HH_S\, |\, \omega\rangle \dd\sphH^{Q-m},\qquad\omega\in\DH^{2n+1-m} .
\]

\begin{lemma}\label{lem607d9517}
	Let $S\subset\HH^n$ be a $C^1_\HH$ submanifold of codimension $m\leq n$ and
	$u:S\to\R$ a Lipschitz function. 
	Let $p\in S$ be fixed and, for $\lambda>0$, let $U,\V,\phi,\phi_\lambda$ and $v_\lambda$ be as in Proposition~\ref{prop607d3565}~\ref{item607d35b3}; define also 
	\begin{align*}
		S_{\lambda} &:= \delta_\lambda(p^{-1}S) , \\
		u_{\lambda}(x) &:=\lambda(u(p\delta_{1/\lambda}x)-u(p)),
	\end{align*}
	so that $v_\lambda(w)=u_\lambda(w\phi_\lambda(w))$.
	Assume that $\lambda_j$ is a sequence such that $\lambda_j\to\infty$ 
	and $v_{\lambda_j}$  converges locally uniformly on $T^\HH_pS$ to  $v: T^\HH_pS\to\R$;
	then
	\[
	\lim_{j\to\infty} u_{\lambda_j} \cur{S_{\lambda_j}} = v\cur{T^\HH_pS} .
	\]
\end{lemma}
\begin{proof}
	We denote by $\areaf_{\phi_\lambda}$ the area factor of $\phi_\lambda$, see~\eqref{eq_areaformula}.
		For every $\omega\in\DH^{2n+1-m}$ and for $j$ large enough  we  have
	\begin{align*}
	u_{\lambda_j} \cur{S_{\lambda_j}}(\omega)
	&= \int_{S_{\lambda_j}} 
		u_{\lambda_j}(x) 
		\langle t^\HH_{S_{\lambda_j}}(x) | \omega(x) \rangle
		\dd \sphH^{Q-m}(x) \\
	&= \int_{\delta_{\lambda_j} U}
		v_{\lambda_j}(w) 
		\langle t^\HH_{S_{\lambda_j}}(w\phi_{\lambda_j}(w)) | \omega(w\phi_{\lambda_j}(w)) \rangle  
		\areaf_{\phi_{\lambda_j}}(w) 
		\dd\sphH^{Q-m}(w).
	\end{align*}
The latter integrand, in $j$, gives a sequence of functions that are supported on some fixed compact subset of $T^\HH_pS$ and  converge uniformly to
	\[
	w\mapsto v(w) \langle t^\HH_{S}(p) | \omega(w) \rangle,
	\]
where we also used Remark~\ref{rem_areafactoris1} together with the fact that $\areaf_{\phi_{\lambda_j}}(w)=\areaf_\phi(\delta_{1/\lambda_j}w)$. This is sufficient to conclude.
\end{proof}

In the following lemma, given a covector $\alpha\in\bwl^1 V_1$ we consider the homogeneous morphism
\begin{equation}\label{eq_Lalfa}
L_\alpha:\HH^n\to\R,\quad L_\alpha(p):=\alpha(p) 
\end{equation}
obtained by identifying $\HH^n$ with $V$ and setting $L_\alpha|_{V_2}:=0$. Observe that $dL_\alpha=\alpha$, where the 1-covector $\alpha$ is identified with a left-invariant 1-form. Moreover, given a $C^1_\HH$ submanifold $S$ of codimension $m< n$  and a 1-form $\alpha$, we denote by $\cur S\hel\alpha$ the Heisenberg $(2n-m)$-current defined by
\[
\cur S\hel\alpha( \omega) =\int_S\langle t^\HH_S|\alpha\wedge\omega\rangle\dd\sphH^{Q-m},\qquad\omega\in\DH^{2n-m}.
\]
Clearly, when $\alpha$ is smooth this is equivalent to $\cur S\hel\alpha( \omega) =\cur S(\alpha \wedge \omega)$; observe that if $\omega\in \DH^{2n-m}$, then $\alpha\wedge \omega \in \DH^{2n+1-m}$ by definition of Heisenberg forms and because $m<n$.

\begin{lemma}\label{lem607d90a0}
	Let $\W\subset\HH^n$ be a homogeneous subgroup of codimension $m< n$.
	Given a measurable $u:\W\to\R$
	and $\alpha\in\bwl^1 V_1$ such that
	$\de_c(u\cur \W)=-\cur\W\hel\alpha$, where we identified the covector $\alpha$ with a left-invariant 1-form.
	Then there exists $c\in\R$ such that $u(w) = c+ L_\alpha(w)$ for $\sphH^{Q-m}$-a.e.~$w\in\W$.
\end{lemma}
\begin{proof}
	If $\alpha=0$ this is a consequence of the  Constancy Theorem in~\cite[Theorem~1.7]{2020arXiv200714286V}.
	If $\alpha\neq 0$, we use the fact that $\de_c\cur\W=0$ (see e.g.~\cite[Proposition~1.9]{2020arXiv200714286V}) to deduce that for every $\omega\in\DH^{2n-m}$
	\begin{align*}
	0=\cur\W(d(L_\alpha\omega))=\cur\W(\alpha\wedge\omega+L_\alpha d\omega)=(\cur\W\hel\alpha)(\omega)+(L_\alpha\cur\W)(d\omega),
	\end{align*}
	i.e., $\de_c(L_\alpha\cur\W)=-\cur\W\hel\alpha$.
	This implies that $\de_c((u-L_\alpha)\cur\W)=0$ and the statement follows from the  Constancy Theorem  again.
\end{proof}

\begin{remark}\label{rem_constifucont}
Clearly, when $u$ is continuous the constant $c$ provided by  Lemma~\ref{lem607d90a0} is $c=u(0)$.
\end{remark}

\begin{lemma}\label{lem607c0430}
	Let $S\subset\HH^n$ be a $C^1_\HH$ submanifold of codimension $m< n$
	 and	$u:S\to\R$ be a Lipschitz function.
	Then there exists a 1-form $\alpha\in L^\infty(S,\bwl^1V_1)$ such that
	\begin{equation}\label{eq607c0436}
		\de_c(u\cur{S}) (\omega)
		= -\cur S\hel\alpha(\omega)\qquad\forall\:\omega\in\DH^{2n-m}\text{ such that }\spt\:\omega\subset\HH^n\setminus\de S.
	\end{equation}
	If $\alpha_1$ and $\alpha_2$ both satisfy~\eqref{eq607c0436}, 
	then $\alpha_1(p)|_{T^\HH_pS} = \alpha_2(p)|_{T^\HH_pS}$, for $\sphH^{Q-m}$-a.e.~$p\in S$.
\end{lemma}
\begin{proof}
	By the McShane-Whitney extension theorem we can extend $u$ to a Lipschitz function $\HH^n\to\R$. 
	Let $(u_j)_j$ be a sequence of smooth functions\footnote{These functions can be easily produced e.g. by group convolution.}  that  converge uniformly to $u$ and such that the Lipschitz constant of $u_j$ is bounded uniformly in $j$.
	Write $d_\HH u_j:=\sum_{i=1}^n(X_iu_j)dx_i+(Y_iu_j)dy_i$; the uniform Lipschitz continuity of $u_j$ implies that $d_\HH u_j$ is uniformly bounded, hence (up to passing to a subsequence) there exists $\alpha\in L^\infty(S;\bwl^1V_1)$ such that $d_\HH u_j$ converges weakly-* to $  \alpha$ in $L^\infty(S;\bwl^1V_1)$. Let us prove that~\eqref{eq607c0436} holds for such $\alpha$.
	
	Let $\omega\in\DH^{2n-m}$  be such that $\spt\:\omega\subset\HH^n\setminus\de S$;  by using   Remark~\ref{rem_boundarylocally0} and a standard partition-of-unity argument one can prove that there exists an open neighborhood $\Omega$ of $\spt\:\omega$ such that $(\de_c\cur S )\hel \Omega=0$.
	Noticing  that $du_j=d_\HH u_j+(Tu_j)\theta$ we have 
	\begin{equation}\label{eq_trattore}
	\begin{split}
	\de_c(u\cur{S})(\omega)
	&= (u\cur{S})(d\omega)
	= \lim_{j\to\infty} (u_j\cur{S})(d\omega)
	= \lim_{j\to\infty} \cur{S}(u_j d\omega)\\
	&= \lim_{j\to\infty} \cur{S}( d(u_j\omega)-du_j\wedge\omega)
	=- \lim_{j\to\infty} \cur{S}(d_\HH u_j\wedge\omega),
	\end{split}
	\end{equation}
where we used the equalities $(\de\cur S )\hel \Omega = (\de_c\cur S )\hel \Omega=0$ and $\omega\wedge\theta=0$. Therefore
\begin{align*}
\de_c(u\cur{S})(\omega)
&=-\lim_{j\to\infty} \int_S \langle t^\HH_S | (\dd_\HH u_j)\wedge\omega \rangle \dd\sphH^{Q-m}
=- \int_S \langle t^\HH_S | \alpha\wedge\omega \rangle \dd\sphH^{Q-m},
\end{align*}
which is~\eqref{eq607c0436}.

As for the last statement, let us introduce the following standard notation: if $t\in\bwl_kV$ and $\alpha\in\bwl^1 V$, then $t\lrcorner\alpha$  denotes the element of $\bwl_{k-1}V$ defined for each $\omega\in \bwl^{k-1}V$ by $\langle t\lrcorner\alpha|\omega \rangle=\langle t|\alpha\wedge\omega \rangle$.
It is now enough to observe that the equality $t^\HH_S\lrcorner(\alpha_1-\alpha_2)=0$ holds  $\sphH^{Q-m}$-a.e. on $S$,  and  the statement follows. 
\end{proof}

\begin{proof}[Proof of Theorem~\ref{thm607bf499}]
Passing to the components of $u:S\to\R^\ell$ separately, we can assume $\ell=1$.

	Let $\alpha$ be as in Lemma~\ref{lem607c0430}.
	Since $\sphH^{Q-m}\hel S $ is locally $(Q-m)$-Ahlfors regular,  for $\sphH^{Q-m}$-a.e. $p\in S$ we have
	\begin{equation}\label{eq_Lebesguepoint}
		\lim_{r\to0^+} \frac{1}{r^{Q-m}}\int_{S\cap B(p,r)}|\alpha-\alpha(p)|\dd\sphH^{Q-m}
	 = 0.
	\end{equation}
	We fix such a $p$ and prove that $u$ is Pansu differentiable along $S$ at $p$ with differential  (recall~\eqref{eq_Lalfa}) $D_\HH^Su_p=L_{\alpha(p)}|_{T^\HH_pS}$, which is uniquely defined  by Lemma~\ref{lem607c0430};  
	this will be enough to conclude. 

	For $\lambda>0$, let $U,\V,\phi,\phi_\lambda$ and $v_\lambda$ be as in Proposition~\ref{prop607d3565}~\ref{item607d35b3}; let also $S_\lambda$ and $u_\lambda$ be as in Lemma~\ref{lem607d9517}.
	By Proposition~\ref{prop607d3565}, we have to prove that $v_\lambda$ converges to $L_{\alpha(p)} $ locally uniformly on $T^\HH_pS$; to this end, we assume that  $\lambda_j\to\infty$ is a sequence such that the functions $v_{\lambda_j}$ converge locally uniformly to some map $v:T^\HH_pS\to\R$ and we prove that $v=L_{\alpha(p)}|T^\HH_pS$. 
	The existence of converging subsequences for the family  $(v_\lambda)_\lambda$ follows from a standard Ascoli-Arzel\`a argument and the uniform continuity of the maps $(\phi_\lambda)_\lambda$, see~\cite[Proposition~3.8]{FS_JGA}.
	For $\omega\in\DH^{2n-m}$ we have
	\begin{equation*}
	\begin{split}
	(u_{\lambda_j}\cur{S_{\lambda_j}})(d\omega)
	&= \int_{S_{\lambda_j}}{\lambda_j} (u(p\delta_{1/{\lambda_j}}x)-u(p))\ \langle t^\HH_{S_{\lambda_j}}(x)|d\omega(x)\rangle\dd\sphH^{Q-m}(x)\\
	&= \lambda_j^{Q-m}\int_S  (u(y)-u(p))\;\langle t^\HH_{S}(y)|{\lambda_j}(d\omega)(\delta_{\lambda_j}(p^{-1}y))\rangle\dd\sphH^{Q-m}(y)\\
	&= \lambda_j^{Q-m}\int_S  (u(y)-u(p))\;\langle t^\HH_{S}(y)|d(\omega\circ L_{p^{-1},\lambda_j})(y)\rangle\dd\sphH^{Q-m}(y),
	\end{split}
	\end{equation*}
	where we set $L_{p^{-1},\lambda}(y):=\delta_{\lambda}(p^{-1}y)$ and used~\eqref{eq_commute}. For large enough $j$ the test form $d(\omega\circ L_{p^{-1},\lambda_j})$ has support in $\HH^n\setminus \de S$: this gives $(\de\cur S)(\omega\circ L_{p^{-1},\lambda_j})=0$, thus
	\begin{equation*}
	\begin{split}
	(u_{\lambda_j}\cur{S_{\lambda_j}})(d\omega)
	&= \lambda_j^{Q-m}\int_S  u(y)\;\langle t^\HH_{S}(y)|d(\omega\circ L_{p^{-1},\lambda_j})(y)\rangle\dd\sphH^{Q-m}(y)\\
	&=\lambda_j^{Q-m}\de(u\cur S)(\omega\circ L_{p^{-1},\lambda_j}).
	\end{split}
	\end{equation*}
	The definition of $\alpha$ (Lemma~\ref{lem607c0430}) yields
	\begin{equation*}
	\begin{split}
	(u_{\lambda_j}\cur{S_{\lambda_j}})(d\omega)
	&= -\lambda_j^{Q-m}\int_S  \langle t^\HH_{S}(y)|\alpha(y)\wedge(\omega\circ L_{p^{-1},\lambda_j})(y)\rangle\dd\sphH^{Q-m}(y)
	\end{split}
	\end{equation*}
	and, if $R>0$ is such that $\spt\:\omega\subset B(0,R)$, we obtain from~\eqref{eq_Lebesguepoint}
	\begin{equation*}
	\begin{split}
	(u_{\lambda_j}\cur{S_{\lambda_j}})(d\omega)
	&= -\lambda_j^{Q-m}\int_{S\cap B(p,R/\lambda_j)}  \langle t^\HH_{S}(y)|\alpha(y)\wedge\omega(\delta_{\lambda_j}(p^{-1}y))\rangle\dd\sphH^{Q-m}(y)\\
	&= -\lambda_j^{Q-m}\int_{S\cap B(p,R/\lambda_j)}  \langle t^\HH_{S}(y)|\alpha(p)\wedge\omega(\delta_{\lambda_j}(p^{-1}y))\rangle\dd\sphH^{Q-m}(y) +o(1).
	\end{split}
	\end{equation*}
	
	We now use Lemma~\ref{lem607d9517} to deduce that, for every test form $\omega\in\DH^{2n-m}$,
	\begin{equation*}
	\begin{split}
	 \de(v\cur{T^\HH_pS})(\omega)
	=\,& v\cur{T^\HH_pS}(d\omega)=\lim_{j\to\infty}u_{\lambda_j}\cur{S_{\lambda_j}}(d\omega)\\
	=\,&- \lim_{j\to\infty} \lambda_j^{Q-m}\int_{S}  \langle t^\HH_{S}(y)|\alpha(p)\wedge\omega(\delta_{\lambda_j}(p^{-1}y))\rangle\dd\sphH^{Q-m}(y)\\
	=\,& - \lim_{j\to\infty}\int_{S_{\lambda_j}}  \langle t^\HH_{S_{\lambda_j}}(x)|\alpha(p)\wedge\omega(x)\rangle\dd\sphH^{Q-m}(x)\\
	=\,& - \lim_{j\to\infty}\int_{\delta_{\lambda_j}U}  \langle t^\HH_{S_{\lambda_j}}(w\phi_{\lambda_j}(w))|\alpha(p)\wedge\omega(w\phi_{\lambda_j}(w))\rangle\areaf_{\phi_{\lambda_j}}(w)\dd\sphH^{Q-m}(w)\\
	=\,& -\int_{T^\HH_pS} \langle t^\HH_{S}(p)|\alpha(p)\wedge\omega\rangle\dd\sphH^{Q-m},
	\end{split}
	\end{equation*}
	where we used Remark~\ref{rem_areafactoris1} and the fact that the area factor verifies $\areaf_{\phi_{\lambda_j}}(w)=\areaf_\phi(\delta_{1/\lambda_j}w)$. We have therefore proved that $\de(v\cur{T^\HH_pS})=-\cur{T^\HH_pS}\hel\alpha(p)$; since $v_\lambda(0)=0$ for every positive $\lambda$, we obtain $v(0)=0$ and  Lemma~\ref{lem607d90a0}  (together with Remark~\ref{rem_constifucont}) implies that $v-v(0)=L_{\alpha(p)}$ on $T^\HH_pS$, as claimed. 
	This implies $\de (u_{\lambda}\cur{S_{\lambda}})\to -\cur{T^\HH_pS}\hel\alpha(p)$, and the proof is accomplished.
\end{proof}

The following result, which we state without proof, is a standard consequence of Theorem~\ref{thm607bf499} together with the Rademacher Theorem for intrinsic Lipschitz graphs in Heisenberg groups~\cite{2020arXiv200714286V}.  
We do not recall here the definition of {\em intrinsic Lipschitz graphs} in Heisenberg groups: see e.g.~\cite{2020arXiv200714286V}. 

\begin{corollary}\label{cor_RademacherLipgr}
Let $\Gamma\subset\HH^n$ be an intrinsic Lipschitz graph of codimension $m<n$ and let $u:\Gamma\to\R^\ell$ be  Lipschitz continuous; then, for $\sphH^{Q-m}$-a.e. $p\in \Gamma$ there exists a homogeneous morphism $L=L(p):\HH^n\to\R^\ell$ such that
\[
\lim_{\substack{q\to p\\q\in \Gamma}} \frac{|u(q) - u(p) -  L(p^{-1}q)|}{d(p,q)} = 0.
\]
Moreover, the restriction $L(p)|_{T^\HH_p\Gamma}$ is uniquely defined.
\end{corollary}

A version  of Theorem~\ref{thm607bf499}  for $\HH$-rectifiable sets reads as follows.

\begin{corollary}\label{cor_RademacherHrectifiablesets}
	Let $R\subset\HH^n$ be countably $\HH$-rectifiable of codimension $m<n$ and let $u:R\to\R^\ell$ be  Lipschitz continuous; then, for $\sphH^{Q-m}$-a.e. $p\in R$ there exists a unique homogeneous morphism $D^R_\HH u_p:T^\HH_pR \to\R^\ell$ such that the following holds. If $\tilde u:\HH^n\to\R^\ell$ is a Lipschitz continuous function such that $\tilde u|_R=u$, then,
	for $\sphH^{Q-m}$-a.e.~$p\in R$,
	\begin{equation}\label{eq_differentiabilityonHrectifiable}
	\lim_{\substack{q\to p\\q\,\in\, p\,T^\HH_pR}} \frac{|\tilde u(q) - \tilde u(p) -  D^R_\HH u_p(p^{-1}q)|}{d(p,q)} = 0.
	\end{equation}
\end{corollary}
\begin{proof}
	Using the notation  of approximate tangent space $T^\HH_pR$ in Definition~\ref{def60af5092}, 
	Theorem~\ref{thm607bf499} claims that,
	for every $i\in\N$, there is a $\sphH^{Q-m}$-null set $N_i\subset S_i$ so that $\tilde u$ is tangentially Pansu differentiable along $S_i$ at every $p\in S_i\setminus N_i$.
	Therefore, for $\sphH^{Q-m}$-a.e.~$p\in R$, there is a $C_\HH^1$-submanifold $S_i$ such that $p\in S_i\setminus N_i$ and $T^\HH_pR=T^\HH_pS_i$.
	Then~\eqref{eq_differentiabilityonHrectifiable} follows from item~\ref{item607d35b9} of Proposition~\ref{prop607d3565}.
\end{proof}

\begin{remark}
In~\eqref{eq_differentiabilityonHrectifiable}, the restriction to points $q$ in the affine tangent plane $p\,T^\HH_pR$ is necessary: this is a phenomenon that occurs also in Euclidean geometry. Consider in fact a sequence $(S_i)_{i\in\N}$ of segments in the plane $\R^2$ such that
\[
\text{$S_0$ joins $(0,0)$ and $(1,0)\qquad$and$\qquad R:=\bigcup_{i\in\N}S_i$ is dense in $\R^2$.}
\]
We can also assume that $\scr H^1(R)<\infty$, so that $R$ is 1-rectifiable. Consider the Lipschitz function $u(x,y)=|y|$; then, the density of $R$ implies that for every $p\in S_0$ there exists no linear map $L:\R^2\to\R$ such that
\[
\lim_{\substack{q\to p\\q\,\in\, R}} \frac{| u(q) -  u(p) -  L(q-p)|}{|q-p|} = 0.
\]
A way to circonvent this problem is to use the notion of \textit{approximate differentiability}.
\end{remark}

%%%%%%%%%%%%%%%%%%%%%%%%%%%%%%%%%%%%%%%%%%%%%%%%%%%%%%%%%%%%%%%
\section{Proof of Theorem~\ref{thm607bf4ef}}\label{sec_proofThmB}
The fundamental tool we use for proving Theorem~\ref{thm607bf4ef} is the Whitney Extension Theorem~\cite[Theorem~6.8]{MR1871966}. We denote by $\HM$ the space of homogeneous morphisms $L:\HH^n\to\R^\ell$ endowed with the natural topology induced (for instance) by the  distance
\[
\rho(L,L'):=\sup\{|L(p)-L'(p)|:p\in B(0,1)\}\quad L,L'\in\HM.
\]
Recall also that, for every $L\in\HM$, there exists a linear map $M_L:\R^{2n}\to\R^\ell$ such that $L(p)=M_L(p_1,\dots,p_{2n})$ for every $p=\exp(p_1X_1+\dots+p_{2n}Y_n+p_{2n+1}T)\in\HH^n$: with this identification, the Whitney Extension Theorem can be written as follows.

\begin{theorem}[{\cite[Theorem 6.8]{MR1871966}}]\label{thm_Whitney}
Let $F\subset\HH^n$ be a closed set and let $u:F\to\R^\ell$ and $L:F\to\HM$ be continuous; assume that for every compact set $K\subset F$
\[
\lim_{r\to 0^+} \sup\left\{\frac{|u(q)-u(p)-L(p)(p^{-1}q)|}{d(p,q)}:p,q\in K, 0<d(p,q)<r \right\}=0.
\]
Then, there exists $\tilde u\in C^1_\HH(\HH^n;\R^\ell)$ such that $\tilde u|_F=u$ and  $D_\HH\tilde u=L$ on $F$.
\end{theorem}

\begin{remark}\label{rem_WhitneyLipschitz}
Although not explicitly stated in~\cite[Theorem~6.8]{MR1871966}, the following fact is a consequence of the construction performed in its proof: if $u$ is Lipschitz continuous on $F$, then the $C^1_\HH$ extension $\tilde u:\HH^n\to\R^\ell$ can be chosen to be also Lipschitz continuous. Moreover, the Lipschitz constant of $\tilde u$ is controlled from above in terms of $n$ and of the Lipschitz constant of $u$ only.
\end{remark}

\begin{proof}[Proof of Theorem~\ref{thm607bf4ef}]
Extend $u$ to a Lipschitz $\R^\ell$-valued function defined on the whole $\HH^n$; by Lemma~\ref{lem_oneSisenough} it is not restrictive to assume that $R$ is actually a $C^1_\HH$ submanifold $S$ of codimension $m$. By Theorem~\ref{thm607bf499} and Lemma~\ref{lem_Borel}, the set $D\subset S$ of points where $u$ is tangentially Pansu differentiable along $S$ is a Borel set such that $\sphH^{Q-m}(S\setminus D)=0$. By the standard Lusin Theorem, there exists a closed set $C\subset D$ such that $\sphH^{Q-m}(S\setminus C)<\varepsilon/2$ and $\nabla_\HH^S u(p)|_C:C\to (V_1)^\ell$ is continuous. Using the notation $q_H$ and $\cdot$  introduced before Lemma~\ref{lem_Borel}, the continuous map $L:C\to\HM$ defined by
\[
L(p)(q):=q_H\cdot\nabla_\HH^S u(p)\qquad\text{for every }p\in C,q\in\HH^n
\]
has the property that, for every $p\in C$,
\begin{equation}\label{eq_limiteperWhitney}
\lim_{\substack{q\to p,\\ q\in C}} \frac{|u(q)-u(p)-L(p)(p^{-1}q)|}{d(p,q)}=0.
\end{equation}
By the Severini-Egorov Theorem, there exists a closed set $F\subset C$ such that $\sphH^{Q-m}(S\setminus F)<\varepsilon$ and the convergence in~\eqref{eq_limiteperWhitney} is uniform on compact subsets of $F$.
To conclude the proof, it suffices to apply Theorem~\ref{thm_Whitney} and recall Remark~\ref{rem_WhitneyLipschitz}.
\end{proof}

%%%%%%%%%%%%%%%%%%%%%%%%%%%%%%%%%%%%%%%%%%%%%%%%%%%%%%%%%%%%%%%
\section{Proof of Theorem~\ref{thm607bf4fe}}\label{sec_proofThmC}
We recall that a homogeneous distance $d$ on $\HH^n$ is {\em rotationally invariant}
 if
\begin{equation}\label{eq:seba1}
d(0,(x,y,t))=d(0,(x',y',t))\qquad\text{whenever }|(x,y)|=|(x',y')|,
\end{equation}
where $|\cdot|$ is the Euclidean norm in $\R^{2n}$.

\begin{proof}[Proof of Theorem~\ref{thm607bf4fe}]
By standard arguments, we can without loss of generality assume that $R$ is a $C^1_\HH$ submanifold $S$ of codimension $m$. By Theorem~\ref{thm607bf4ef}, for every positive integer $i$ there exists $g_i\in C^1_\HH(\HH^n;\R^\ell)$ such that
\[
\sphH^{Q-m}(B_i)<2^{-i-1},\quad\text{where }B_i:=\{p\in S: u(p)\neq g_i(p)\text{ or }D_\HH^Su(p)\neq D_\HH^S{g_i}(p)\}.
\]
Moreover, by Remark~\ref{rem_WhitneyLipschitz} we can assume that the Lipschitz constants of $g_i$ are uniformly bounded. Let $C_i:=\cup_{j\geq i}B_j\subset S$ and $C_\infty:=\cap_{i} C_i$; observe that $\sphH^{Q-m}(C_i)<2^{-i}$ and $\sphH^{Q-m}(C_\infty)=0$. By the coarea formula  in~\cite[Theorem 1.7]{JNGV} we obtain for every Borel function $h:S\to[0,+\infty)$
\begin{align*}
	&\int_S \chi_{S\setminus C_i}(p)h(p)\coarea(T^\HH_pS,D_\HH^S {g_i}_p) \, \dd\sphH^{Q-m} (p)\\
	= & \int_{\R^\ell} \int_{S\cap g_i^{-1}(s)} \chi_{S\setminus C_i}h\dd\sphH^{Q-m-\ell}\,\dd\mathscr L^\ell(s),
\end{align*}
where $\chi_{S\setminus C_i}$ is the characteristic function of $S\setminus C_i$ (which is a Borel subset of $S$) and $\coarea$ denotes the (continuous) coarea factor introduced in~\cite[Proposition~4.5]{JNGV}. The previous formula is the same as
\begin{align*}
	\int_{S\setminus C_i}h(p)\coarea(T^\HH_pS,D_\HH^S u_p) \, \dd\sphH^{Q-m} (p)
	= \int_{\R^\ell} \int_{(S\setminus C_i)\cap u^{-1}(s)} h\dd\sphH^{Q-m-\ell}\,\dd\mathscr L^\ell(s).
\end{align*}
Recalling that $\sphH^{Q-m}(C_\infty)=0$ and that $S\setminus C_i\nearrow S\setminus C_\infty$ as $i\to\infty$,  by monotone convergence we obtain
\begin{align*}
 \int_{S}h(p)\coarea(T^\HH_pS,D_\HH^S u_p) \, \dd\sphH^{Q-m} (p) = &
	\int_{S\setminus C_\infty}h(p)\coarea(T^\HH_pS,D_\HH^S u_p) \, \dd\sphH^{Q-m} (p)\\
	= & \int_{\R^\ell} \int_{(S\setminus C_\infty)\cap u^{-1}(s)} h\dd\sphH^{Q-m-\ell}\,\dd\mathscr L^\ell(s)\\
		= & \int_{\R^\ell} \int_{S\cap u^{-1}(s)} h\dd\sphH^{Q-m-\ell}\,\dd\mathscr L^\ell(s).
\end{align*}
In the last equality we used the fact that $\sphH^{Q-m-\ell}( C_\infty\cap u^{-1}(s))=0$ for $\mathscr L^\ell$-a.e. $s\in\R^\ell$: this is a consequence of the
{\it coarea inequality} (see e.g.~\cite[Lemma~4.3]{JNGV} and the references therein), which implies that for a suitable $K>0$
\[
\int_{\R^\ell} \sphH^{Q-m-\ell}( C_\infty\cap u^{-1}(s))\,\dd\mathscr L^\ell(s) \leq K\sphH^{Q-m}(C_\infty)=0.
\]

In order to prove the last statement in Theorem~\ref{thm607bf4fe}, it is enough to reason as above and use the coarea formula proved  in~\cite[Theorem 1.7]{JNGV} for rotationally invariant distances.
This concludes the proof.
\end{proof}

%%%%%%%%%%%%%%%%%%%%%%%%%%%%%%%%%%%%%%%%%%%%%%%%%%%%%%%%%%%%%%%
%%%%%%%%%%%%%%%%%%%%%%%%%%%%%%%%%%%%%%%%%%%%%%%%%%%%%%%%%%%%%%%

\printbibliography
\end{document}